\documentclass[reqno]{amsart}

\usepackage[latin1]{inputenc}
\usepackage{amssymb}
\usepackage{graphicx}
\usepackage{amscd}
\usepackage[hidelinks]{hyperref}
\usepackage{color}
\usepackage{float}
\usepackage{graphics,amsmath,amssymb}
\usepackage{amsthm}
\usepackage{amsfonts}
\usepackage{latexsym}
\usepackage{epsf}
\usepackage{enumerate}
\usepackage{xifthen}
\usepackage{mathrsfs}
\usepackage{dsfont}
\usepackage{makecell}
\usepackage[FIGTOPCAP]{subfigure}
\usepackage{amsmath}
\allowdisplaybreaks[4]
\usepackage{listings}
\usepackage{etoolbox}
\usepackage{fancyhdr}

 \newtheoremstyle{mytheorem}% name % cf. thmtest.tex of AMSLaTeX
 {3pt}%      Space above
 {3pt}%      Space below
 {\slshape}% Body font
 {}%         Indent amount (empty = no indent,
 {\bfseries}% Thm head font
 {.}%        Punctuation after thm head
 { }%        Space after thm head: " " = normal interword space;
 {}%         Thm head spec (can be left empty, meaning `normal')

\numberwithin{equation}{section}

\theoremstyle{theorem}
\newtheorem{theorem}{Theorem}[section]
\newtheorem{corollary}[theorem]{Corollary}
\newtheorem{lemma}[theorem]{Lemma}

\theoremstyle{definition}

\newcommand{\Keywords}[1]{\ifthenelse{\isempty{#1}}{}{\smallskip \smallskip \noindent \textbf{Keywords}. #1}}
\newcommand{\MSC}[2][2010]{\ifthenelse{\isempty{#2}}{}{\smallskip \smallskip \noindent \textbf{#1MSC}. #2}}
\newcommand{\abstractnote}[1]{\ifthenelse{\isempty{#1}}{}{\smallskip \smallskip \noindent \textsuperscript{\dag}#1}}

\makeatletter
\def\specialsection{\@startsection{section}{1}%
  \z@{\linespacing\@plus\linespacing}{.5\linespacing}%
  {\normalfont}}% NEW
\def\section{\@startsection{section}{1}%
  \z@{.7\linespacing\@plus\linespacing}{.5\linespacing}%
  {\normalfont\scshape}}% NEW
\patchcmd{\@settitle}{\uppercasenonmath\@title}{\Large\boldmath}{}{}
\patchcmd{\@settitle}{\begin{center}}{\begin{flushleft}}{}{}
\patchcmd{\@settitle}{\end{center}}{\end{flushleft}}{}{}
\patchcmd{\@setauthors}{\MakeUppercase}{\normalsize}{}{}
\patchcmd{\@setauthors}{\centering}{\raggedright}{}{}
\patchcmd{\section}{\scshape}{\large\bfseries\boldmath}{}{}
\patchcmd{\subsection}{\bfseries}{\bfseries\boldmath}{}{}
\renewcommand{\@secnumfont}{\bfseries}
\patchcmd{\@startsection}{\@afterindenttrue}{\@afterindentfalse}{}{}
\patchcmd{\abstract}{\leftmargin3pc}{\leftmargin1pc}{}{}

\def\maketitle{\par
  \@topnum\z@ % this prevents figures from falling at the top of page 1
  \@setcopyright
  \thispagestyle{empty}% this sets first page specifications
  \ifx\@empty\shortauthors \let\shortauthors\shorttitle
  \else \andify\shortauthors
  \fi
  \@maketitle@hook
  \begingroup
  \@maketitle
  \toks@\@xp{\shortauthors}\@temptokena\@xp{\shorttitle}%
  \toks4{\def\\{ \ignorespaces}}% defend against questionable usage
  \edef\@tempa{%
    \@nx\markboth{\the\toks4
      \@nx\MakeUppercase{\the\toks@}}{\the\@temptokena}}%
  \@tempa
  \endgroup
  \c@footnote\z@
  \@cleartopmattertags
}
\makeatother

\newcommand{\bF}{\mathbf{F}}
\newcommand{\ov}[1]{\overline{#1}}
\newcommand{\ee}[1]{e\left(#1\right)}
\newcommand{\abs}[1]{\left|#1\right|}
\newcommand{\Leg}[2]{\left(\frac{#1}{#2}\right)}

\newcommand{\sol}{\mathrm{Sol}}

\title[On the power mean of a sum]{On the power mean of a sum analogous to the Kloosterman sum}

\author[S. Chern]{Shane Chern}
\address{Department of Mathematics, The Pennsylvania State University, University Park, PA 16802, USA}
\email{shanechern@psu.edu}

\date{}

\begin{document}

%{\footnotesize\noindent \textit{Preprint}, \arxiv{1710.08507}.}
%
%\bigskip \bigskip

\maketitle

\begin{abstract}

In this paper, we study the power mean of a sum analogous to the Kloosterman sum by using analytic methods for character sums.


\MSC{11L05.}
\end{abstract}

\section{Introduction}

Throughout this paper, let $p$ be an odd prime. Let $\chi_0$ be the principal Dirichlet character mod $p$. We occasionally use $\chi_L(n)$ to denote the Legendre symbol $(n|p)$ mod $p$. Let $a\in \bF_p^\times$, we write $\ov{a}$ the multiplicative inverse of $a$ in $\bF_p^\times$ (i.e.~$\ov{a}a\equiv 1 \bmod{p}$). We also write $\ov{0}=0$ for notational convenience. Finally, we denote $e(y):=e^{2\pi i y}$.

The Gauss sum associated with a Dirichlet character $\chi$ is defined by
$$G(n,\chi):=\sum_{a=1}^{p}\chi(a)\ee{\frac{na}{p}}.$$
We will frequently use the following classical property of $G(n,\chi)$ in the sequel:
\begin{equation}\label{eq:1.1}
\sum_{a=1}^{p}\chi(a)\ee{\frac{na}{p}}=\ov{\chi}(n)\sum_{a=1}^{p}\chi(a)\ee{\frac{a}{p}},
\end{equation}
except for the case $\chi=\chi_0$ and $n\equiv 0 \bmod{p}$ holding simultaneously. We also use the conventional notation $\tau(\chi):=G(1,\chi)$. The interested readers may refer to  Apostol's book \cite{Apo1976} for details.

Among various character sums, the generalized Kloosterman sum $K(m,n,\chi;p)$, which is defined as follows, was intensively studied by many authors (see \cite{Cho1967,Est1961,LL2013,Shp2007,Ye1999,Zha2016})
$$K(m,n,\chi;p):=\sum_{a=1}^p \chi(a)\ee{\frac{ma+n\ov{a}}{p}}.$$
For example, Zhang \cite{Zha2016} proved the fourth power mean formula of $K(m,n,\chi;p)$ for $n$ coprime to $p$
$$\sum_{m=1}^{p-1}\abs{\sum_{a=1}^{p-1}\chi(a)\ee{\frac{ma+n\ov{a}}{p}}}^4=
\begin{cases}
2p^3-3p^2-3p-1 & \text{if $\chi=\chi_0 \bmod{p}$},\\
3p^3-8p^2 & \text{if $\chi=\chi_L \bmod{p}$},\\
p^2(2p-7) & \text{otherwise}.
\end{cases}$$

In a recent paper, Lv and Zhang \cite{LZ2017} considered a new sum, which can be regarded as a companion of the Kloosterman sum. The sum is defined by
$$H(m,n,k,\chi;p):=\sum_{a=1}^{p-1} \chi(ma+n\ov{a})\ee{\frac{ka}{p}}.$$
They also obtained a computational formula for the following hybrid power mean
$$\sum_{\chi \bmod{p}}\abs{\sum_{a=1}^{p-1}\chi(a)\ee{\frac{ma^2}{p}}}^2\abs{\sum_{b=1}^{p-1}\chi(b+\ov{b})\ee{\frac{nb}{p}}}^2$$
for $m$ coprime to $p$.

In this paper, we shall have a further investigation on $H(m,n,k,\chi;p)$.

\begin{theorem}\label{th:1}
Let $p$ be an odd prime. Let $n$ and $k$ be integers coprime to $p$. We have
\begin{equation}
\sum_{\chi \bmod{p}} \abs{\sum_{m=1}^{p-1}\abs{\sum_{a=1}^{p-1} \chi(ma+n\ov{a})\ee{\frac{ka}{p}}}^2}^2=(p-1)(p^4-7p^3+17p^2-5p-25).
\end{equation}
\end{theorem}

In fact, our ultimate goal is to find the fourth power mean of $H(m,n,k,\chi;p)$.

\begin{theorem}\label{th:2}
Let $p>3$ be an odd prime. Let $n$ and $k$ be integers coprime to $p$. We have
\begin{equation}
\sum_{\chi \bmod{p}} \sum_{m=1}^{p-1} \abs{\sum_{a=1}^{p-1} \chi(ma+n\ov{a})\ee{\frac{ka}{p}}}^4=(p-1)\Big(p^2-10p+37+p\,T(p)-T_L(p)\Big),
\end{equation}
where $T(p)$ and $T_{L}(p)$ are defined by
\begin{equation*}
T(p)=\underset{(a-1)^2(u\ov{a}-1)(ub-1)\equiv (b-1)^2(u\ov{b}-1)(ua-1) \bmod{p}}{\sum_{u=1}^{p-1}\sum_{a=1}^{p-1}\sum_{b=1}^{p-1}} \chi_0((a-1)^2(u\ov{a}-1)(ub-1))
\end{equation*}
and
\begin{equation*}
T_L(p)=\sum_{u=1}^{p-1}\left(\sum_{a=1}^{p-1} \Leg{(ua-1)(u\ov{a}-1)}{p}\right)^2.
\end{equation*}
\end{theorem}

However, we remark that $T(p)$ and $T_L(p)$ seem to have no closed forms, at least numerically. On the other hand, we still can obtain an equivalent expression of $T(p)$ that involves merely the Legendre symbol mod $p$.

\begin{theorem}\label{th:22}
Let $p>3$ be an odd prime. We have
\begin{equation}
T(p)=(p-5)(2p-5)-\Leg{2}{p}+\underset{\substack{a\not\equiv b \bmod{p}\\a+b-2ab\not\equiv 0 \bmod{p}}}{\sum_{a=1}^{p-1}\sum_{b=1}^{p-1}}\Leg{\Delta(a,b)}{p},
\end{equation}
where
$$\Delta(a,b)=\left(1-a^2b^2\right)^2+4ab(a+b-2)(a+b-2ab).$$
\end{theorem}

Finally, it is worth mentioning that when proving Theorem \ref{th:2}, we obtain a curious character sum identity, which has its independent interest.

\begin{theorem}\label{th:3}
Let $p>3$ be an odd prime. We have
\begin{equation}
\sum_{\psi \bmod{p}} \abs{\sum_{u=1}^{p-1} \sum_{a=1}^{p-1} \ov{\psi}(ua-1)\psi(u\ov{a}-1)\psi^2(a-1)}^2=(p-1)(p^3-8p^2+29p-53).
\end{equation}
\end{theorem}

\section{Proof of Theorem \ref{th:1}}

To prove Theorem \ref{th:1}, it suffices to compute the value of
\begin{equation}\label{eq:2.1}
\sum_{m=1}^{p-1}\abs{\sum_{a=1}^{p-1} \chi(ma+n\ov{a})\ee{\frac{ka}{p}}}^2
\end{equation}
for each character $\chi$ mod $p$. Note that the value of \eqref{eq:2.1} remains the same for any choices of $n$ and $k$ provided that they are coprime to $p$.% In the sequel, we will frequently use this property for various sums, which can be checked easily.

We now show

\begin{theorem}\label{th:2.1}
Let $p$ be an odd prime. Let $n$ and $k$ be integers coprime to $p$. We have
\begin{equation}
\sum_{m=1}^{p-1}\abs{\sum_{a=1}^{p-1} \chi(ma+n\ov{a})\ee{\frac{ka}{p}}}^2=\begin{cases}
2p-5 & \text{if $\chi=\chi_0$},\\
p(p-3-\ov{\chi}(-1)) & \text{otherwise}.
\end{cases}
\end{equation}
\end{theorem}

\begin{proof}
If $\chi=\chi_0$, we have
$$\sum_{m=1}^{p-1}\abs{\sum_{a=1}^{p-1} \chi_0(ma+n\ov{a})\ee{\frac{ka}{p}}}^2=\sum_{m=1}^{p-1}\abs{\sum_{a=1}^{p-1} \chi_0(ma-\ov{a})\ee{\frac{a}{p}}}^2.$$
Note that $ma-\ov{a}\equiv 0 \pmod{p}$ is equivalent to $a^2\equiv \ov{m} \pmod{p}$ for $a\in \bF_p^\times$. Hence
\begin{align*}
&\sum_{m=1}^{p-1}\abs{\sum_{a=1}^{p-1} \chi_0(ma-\ov{a})\ee{\frac{a}{p}}}^2\\
&\quad=\sum_{\substack{m=1\\\Leg{\ov{m}}{p}=-1}}^{p-1}\abs{\sum_{a=1}^{p-1} \ee{\frac{a}{p}}}^2+\sum_{\substack{m=1\\\Leg{\ov{m}}{p}=1}}^{p-1}\abs{\sum_{\substack{a=1\\a^2\not\equiv \ov{m} \bmod{p}}}^{p-1} \ee{\frac{a}{p}}}^2\\
&\quad=\frac{p-1}{2}+\sum_{\substack{m=1\\\Leg{m}{p}=1}}^{p-1}\abs{\sum_{\substack{a=1\\a^2\not\equiv m \bmod{p}}}^{p-1} \ee{\frac{a}{p}}}^2\\
&\quad=\frac{p-1}{2}+\sum_{\substack{m=1\\\Leg{m}{p}=1}}^{p-1}\abs{-1-\ee{\frac{m_0}{p}}-\ee{\frac{-m_0}{p}}}^2 \tag{where $m_0^2\equiv m \bmod{p}$}\\
&\quad=\frac{p-1}{2}+\frac{1}{2}\sum_{m_0=1}^{p-1}\abs{-1-\ee{\frac{m_0}{p}}-\ee{\frac{-m_0}{p}}}^2\\
&\quad=\frac{p-1}{2}+\frac{1}{2}\sum_{m_0=1}^{p-1}\left(3+2\ee{\frac{m_0}{p}}+2\ee{\frac{-m_0}{p}}+\ee{\frac{2m_0}{p}}+\ee{\frac{-2m_0}{p}}\right)\\
&\quad =2p-5.
\end{align*}

If $\chi$ is non-principal, we have
\begin{align*}
\sum_{m=1}^{p-1}\abs{\sum_{a=1}^{p-1} \chi(ma+n\ov{a})\ee{\frac{ka}{p}}}^2&=\sum_{m=0}^{p-1}\abs{\sum_{a=1}^{p-1} \chi(ma+\ov{a})\ee{\frac{a}{p}}}^2-\abs{\sum_{a=1}^{p-1} \chi(\ov{a})\ee{\frac{a}{p}}}^2\\
&=\sum_{m=0}^{p-1}\abs{\sum_{a=1}^{p-1} \chi(ma+\ov{a})\ee{\frac{a}{p}}}^2-p,
\end{align*}
since $|\tau(\ov{\chi})|=\sqrt{p}$. With the help of \eqref{eq:1.1}, we have
\begin{align*}
&\sum_{m=0}^{p-1}\abs{\sum_{a=1}^{p-1} \chi(ma+\ov{a})\ee{\frac{a}{p}}}^2\\
&\quad = \sum_{m=0}^{p-1}\abs{\sum_{a=1}^{p-1} \frac{1}{\tau(\ov{\chi})}\sum_{u=1}^{p-1}\ov{\chi}(u)\ee{\frac{u(ma+\ov{a})}{p}}\ee{\frac{a}{p}}}^2\\
&\quad = \frac{1}{p}\sum_{u=1}^{p-1}\sum_{v=1}^{p-1}\ov{\chi}(u\ov{v})\sum_{a=1}^{p-1}\sum_{b=1}^{p-1}\sum_{m=0}^{p-1}\ee{\frac{m(ua-vb)+(a-b+u\ov{a}-v\ov{b})}{p}}\\
&\quad = \sum_{u=1}^{p-1}\sum_{v=1}^{p-1}\ov{\chi}(u\ov{v})\underset{ua\equiv vb \bmod{p}}{\sum_{a=1}^{p-1}\sum_{b=1}^{p-1}}\ee{\frac{a-b+u\ov{a}-v\ov{b}}{p}}\\
&\quad = \sum_{u=1}^{p-1}\sum_{v=1}^{p-1}\ov{\chi}(u\ov{v})\sum_{a=1}^{p-1}\ee{\frac{a-u\ov{v}a+u\ov{a}-\ov{u}v^2\ov{a}}{p}}\\
&\quad = \sum_{u=1}^{p-1}\sum_{v=1}^{p-1}\ov{\chi}(u)\sum_{a=1}^{p-1}\ee{\frac{a-ua+uv\ov{a}-\ov{u}v\ov{a}}{p}}\\
&\quad = \sum_{u=1}^{p-1}\ov{\chi}(u)\sum_{a=1}^{p-1}\ee{\frac{a-ua}{p}}\sum_{v=1}^{p-1}\ee{\frac{v(u\ov{a}-\ov{u}\ov{a})}{p}}\\
&\quad = (p-1)\sum_{u=\pm 1}\ov{\chi}(u)\sum_{a=1}^{p-1}\ee{\frac{a-ua}{p}}-\sum_{u=2}^{p-2}\ov{\chi}(u)\sum_{a=1}^{p-1}\ee{\frac{a(1-u)}{p}}\\
&\quad = (p-1)^2-\ov{\chi}(-1)(p-1)+\sum_{u=2}^{p-2}\ov{\chi}(u)\\
&\quad = (p-1)^2-\ov{\chi}(-1)(p-1) -(1+\ov{\chi}(-1))\\
&\quad = p(p-2-\ov{\chi}(-1)).
\end{align*}
This implies
$$\sum_{m=1}^{p-1}\abs{\sum_{a=1}^{p-1} \chi(ma+n\ov{a})\ee{\frac{ka}{p}}}^2=p(p-2-\ov{\chi}(-1))-p=p(p-3-\ov{\chi}(-1)),$$
as desired.
\end{proof}

From Theorem \ref{th:2.1}, it follows that
\begin{align*}
&\sum_{\chi \bmod{p}} \abs{\sum_{m=1}^{p-1}\abs{\sum_{a=1}^{p-1} \chi(ma+n\ov{a})\ee{\frac{ka}{p}}}^2}^2\\
&\quad=(2p-5)^2+\frac{p-3}{2}(p(p-4))^2+\frac{p-1}{2}(p(p-2))^2\\
&\quad=(p-1)(p^4-7p^3+17p^2-5p-25).
\end{align*}

\section{Proofs of Theorems \ref{th:22} and \ref{th:3}}

In this section, we first prove Theorem \ref{th:3}. Given an arbitray $N\in\bF_p^\times$, we define
\begin{align*}
\mathcal{S}(N)=\bigg\{(u,a)\in \ &\bF_p^\times \times(\bF_p^\times-\{ 1\}):\\
&u \not\equiv a,\ov{a} \text{ and } \ov{(ua-1)}(u\ov{a}-1)(a-1)^2\equiv N\bigg\}.
\end{align*}

In the next Lemma, we characterize each $\mathcal{S}(N)$.

\begin{lemma}\label{le:3.1}
Let
\begin{align*}
\mathcal{S}^*(N)=\bigg\{(u,a)\in\ & \bF_p^\times \times(\bF_p^\times-\{\pm1\}):\\
&u\equiv \left((a-1)^2-N\right)\ov{\left((a-1)^2\ov{a}-aN\right)}\\
&\text{with } (a-1)^2\not\equiv N,\ (\ov{a}-1)^2\not\equiv N\bigg\}.
\end{align*}
Then $\mathcal{S}(N)=\mathcal{S}^*(N)$ for $N\not\equiv 4$ and
$$\mathcal{S}(4)=\mathcal{S}^*(4)\cup \left\{(u,-1):u\in\bF_p^\times-\{-1\}\right\}.$$
\end{lemma}

\begin{proof}
\textit{Case 1}: $a\equiv -1$. We have
\begin{align*}
\ov{(ua-1)}(u\ov{a}-1)(a-1)^2\equiv 4(-u-1)\ov{(-u-1)}\equiv \begin{cases}
4 & \text{if $u\not\equiv -1$},\\
0 & \text{if $u\equiv -1$}.
\end{cases}
\end{align*}
Hence
$$\left\{(u,-1):u\in\bF_p^\times-\{-1\}\right\}\subseteq \mathcal{S}(4).$$

\textit{Case 2}: $a\not\equiv -1$. Note that
$$\ov{(ua-1)}(u\ov{a}-1)(a-1)^2\equiv N$$
is equivalent to
\begin{equation}\label{eq:3.1}
\left((a-1)^2 \ov{a}-aN\right)u\equiv (a-1)^2-N.
\end{equation}

I claim that \eqref{eq:3.1} has at most one solution $u$ in $\bF_p^\times$ for each fixed $a$. Otherwise, we have two simultaneous congruences
$$\begin{cases}
(a-1)^2 \ov{a}-aN\equiv 0,\\
(a-1)^2-N \equiv 0.
\end{cases}$$
It follows that $a\equiv \pm 1$, which is a contradiction.

I also claim that $u\equiv a$ or $\ov{a}$ cannot be a solution to \eqref{eq:3.1}. If $u\equiv a$ is a solution, then
$$(a-1)^2-a^2N\equiv (a-1)^2-N,$$
implying that $a\equiv \pm 1$. A contradiction. Similarly, if $u\equiv \ov{a}$ is a solution, then
$$(\ov{a}-1)^2-N\equiv (a-1)^2-N.$$
We also obtain $a\equiv \pm 1$, which contradicts our assumption.

Hence
$$\left(\left((a-1)^2-N\right)\ov{\left((a-1)^2\ov{a}-aN\right)},a\right)\in\mathcal{S}(N)$$
provided $a\in\bF_p^\times-\{\pm1\}$, $(a-1)^2\not\equiv N$ and $(\ov{a}-1)^2\not\equiv N$.
\end{proof}

As a direct consequence of Lemma \ref{le:3.1}, we have

\begin{corollary}
We have
$$|\mathcal{S}(N)|=\begin{cases}
p-5 & \text{if $N=1$},\\
2p-7 & \textit{if $N=4$},\\
p-7 & \text{if $\Leg{N}{p}=1$ and $N\ne 1,4$},\\
p-3 & \text{if $\Leg{N}{p}=-1$}.
\end{cases}$$
\end{corollary}

Finally, it follows that
\begin{align*}
&\sum_{\psi \bmod{p}} \abs{\sum_{u=1}^{p-1} \sum_{a=1}^{p-1} \ov{\psi}(ua-1)\psi(u\ov{a}-1)\psi^2(a-1)}^2\\
&\quad =(p-1)\underset{\ov{(ua-1)}(u\ov{a}-1)(a-1)^2\equiv \ov{(vb-1)}(v\ov{b}-1)(b-1)^2}{\sum_{u=1}^{p-1}\sum_{v=1}^{p-1}\sum_{a=1}^{p-1}\sum_{b=1}^{p-1}}\chi_0\left(\ov{(ua-1)}(u\ov{a}-1)(a-1)^2\right)\\
&\quad =(p-1)\sum_{N=1}^{p-1}\underset{(u,a),(v,b)\in \mathcal{S}(N)}{\sum_{u=1}^{p-1}\sum_{v=1}^{p-1}\sum_{a=1}^{p-1}\sum_{b=1}^{p-1}}\ 1\\
&\quad =(p-1)\sum_{N=1}^{p-1} |\mathcal{S}(N)|^2\\
&\quad =(p-1)\left((p-5)^2+(2p-7)^2+\frac{p-5}{2}(p-7)^2+\frac{p-1}{2}(p-3)^2\right)\\
&\quad =(p-1)(p^3-8p^2+29p-53).
\end{align*}
We therefore complete the proof of Theorem \ref{th:3}.

We next prove Theorem \ref{th:22}. Define
\begin{align*}
\mathcal{U}=\Bigg\{(u,a,b)\in\  &\bF_p^\times \times \bF_p^\times \times \bF_p^\times:\\
& (ua-1)(u\ov{b}-1)(b-1)^2\equiv (ub-1)(u\ov{a}-1)(a-1)^2\Bigg\}
\end{align*}
and
\begin{align*}
\mathcal{U}_0=\bigg\{(u,a,b)\in\  &\bF_p^\times \times \bF_p^\times \times \bF_p^\times:\\
& (ua-1)(u\ov{b}-1)(b-1)^2\equiv (ub-1)(u\ov{a}-1)(a-1)^2\equiv 0\bigg\}.
\end{align*}

One readily sees that
\begin{align*}
T(p) = \underset{(a-1)^2(u\ov{a}-1)(ub-1)\equiv (b-1)^2(u\ov{b}-1)(ua-1) \bmod{p}}{\sum_{u=1}^{p-1}\sum_{a=1}^{p-1}\sum_{b=1}^{p-1}} \chi_0((a-1)^2(u\ov{a}-1)(ub-1))
\end{align*}
equals $\abs{\mathcal{U}}-\abs{\mathcal{U}_0}$.

We first characterize $\mathcal{U}_0$. In fact, given any $(u,a,b)\in \mathcal{U}_0$, it is of one of the following forms
$$\begin{array}{lll}
(1,1,n), & (n,1,n), & (n,1,1),\\
(1,n,1), & (n,n,1), & (-1,-1,n),\\
(-1,n,-1), & (n,n,n), & (n,\ov{n},\ov{n}). 
\end{array}$$
Hence we deduce
$$\abs{\mathcal{U}_0}=9(p-3)+7=9p-20.$$

To compute $\abs{\mathcal{U}}$, we notice that $(u,a,a)$ belongs to $\mathcal{U}$ for any $u$ and $a$ in $\bF_p^\times$. This case contributes $(p-1)^2$ elements in $\mathcal{U}$. Hence it suffices to count elements $(u,a,b)\in \mathcal{U}$ with $a\not\equiv b$. In this case
$$(ua-1)(u\ov{b}-1)(b-1)^2\equiv (ub-1)(u\ov{a}-1)(a-1)^2$$
is equivalent to
\begin{equation}\label{eq:cong}
(a+b-2ab)u^2-(1+ab)(1-ab)u+ab(2-a-b)\equiv 0.
\end{equation}

Now for fixed $a$ and $b$ in $\bF_p^\times$, we denote by $\sol(a,b)$ the set of solutions $u\in \bF_p$ to \eqref{eq:cong}. Then
\begin{align*}
&\abs{\left\{(u,a,b)\in \mathcal{U} : a\not\equiv b\right\}}\\
&\quad= \underset{a\not\equiv b \bmod{p}}{\sum_{a=1}^{p-1}\sum_{b=1}^{p-1}}\abs{\sol(a,b)}-\abs{\left\{(a,b)\in \bF_p^\times \times \bF_p^\times : a\not\equiv b,\ 0\in\sol(a,b)\right\}}.
\end{align*}

If $0\in\sol(a,b)$, then $a+b\equiv 2$. The assumptions $a,b\in\bF_p^\times$ and $a\not\equiv b$ imply that $a\not\equiv 0,1,2$. Hence
$$\abs{\left\{(a,b)\in \bF_p^\times \times \bF_p^\times : a\not\equiv b,\ 0\in\sol(a,b)\right\}}=p-3.$$

Now we compute
$$\underset{a\not\equiv b \bmod{p}}{\sum_{a=1}^{p-1}\sum_{b=1}^{p-1}}\abs{\sol(a,b)}.$$

If \eqref{eq:cong} is a quadratic congruence of $u$, we have $a+b-2ab\not\equiv 0$. This case contributes
\begin{align*}
&\underset{\substack{a\not\equiv b \bmod{p}\\a+b-2ab\not\equiv 0 \bmod{p}}}{\sum_{a=1}^{p-1}\sum_{b=1}^{p-1}}\left(1+\Leg{\Delta(a,b)}{p}\right)\\
&\quad=\underset{\substack{a\not\equiv b \bmod{p}\\a+b-2ab\not\equiv 0 \bmod{p}}}{\sum_{a=1}^{p-1}\sum_{b=1}^{p-1}}\ 1+\underset{\substack{a\not\equiv b \bmod{p}\\a+b-2ab\not\equiv 0 \bmod{p}}}{\sum_{a=1}^{p-1}\sum_{b=1}^{p-1}}\Leg{\Delta(a,b)}{p}\\
&\quad=2(p-2)+(p-3)^2+\underset{\substack{a\not\equiv b \bmod{p}\\a+b-2ab\not\equiv 0 \bmod{p}}}{\sum_{a=1}^{p-1}\sum_{b=1}^{p-1}}\Leg{\Delta(a,b)}{p}\\
&\quad=p^2-4p+5+\underset{\substack{a\not\equiv b \bmod{p}\\a+b-2ab\not\equiv 0 \bmod{p}}}{\sum_{a=1}^{p-1}\sum_{b=1}^{p-1}}\Leg{\Delta(a,b)}{p},
\end{align*}
where $\Delta(a,b)=\left((1+ab)(1-ab)\right)^2-4(a+b-2ab)ab(2-a-b)$.

If \eqref{eq:cong} is not a quadratic congruence of $u$, we have $a+b-2ab\equiv 0$, i.e.
$$(a,b)\equiv \left(a,a\ov{(2a-1)}\right)$$
provided $a\not\equiv \ov{2}$. In this case, I claim that \eqref{eq:cong} has at most one solution $u\in \bF_p$ for fixed $a$ and $b$. Otherwise, we have
$$\begin{cases}
a+b-2ab\equiv 0\\
(1+ab)(1-ab) \equiv 0\\
ab(2-a-b)\equiv 0
\end{cases}$$
holding simultaneously. This implies that $(a,b)\equiv (1,1)$, violating the assumption $a\not\equiv b$. I also claim that \eqref{eq:cong} has a solution only if $a\not\equiv 0,1,\ov{2}$ and $(a+1)^2\not\equiv 2$. Here $a\not\equiv 0,1$ comes from the assumptions $a\in\bF_p^\times$ and $a\not\equiv b$. Note also that $a+b-2ab\equiv 0$ implies $a\not\equiv \ov{2}$. Finally, we obtain $(a+1)^2\not\equiv 2$ from $(1+ab)(1-ab)\not\equiv 0$. Note that $(a+1)^2\equiv 2$ has either two distinct solutions or no solutions in $\bF_p^\times$. Furthermore, $0$, $1$ and $\ov{2}$ are not solutions to $(a+1)^2\equiv 2$. Hence this case contributes
$$p-4-\Leg{2}{p}.$$

Hence
\begin{align*}
\underset{a\not\equiv b \bmod{p}}{\sum_{a=1}^{p-1}\sum_{b=1}^{p-1}}\abs{\sol(a,b)}&=p-4-\Leg{2}{p}+p^2-4p+5+\underset{\substack{a\not\equiv b \bmod{p}\\a+b-2ab\not\equiv 0 \bmod{p}}}{\sum_{a=1}^{p-1}\sum_{b=1}^{p-1}}\Leg{\Delta(a,b)}{p}\\
&=p^2-3p+1-\Leg{2}{p}+\underset{\substack{a\not\equiv b \bmod{p}\\a+b-2ab\not\equiv 0 \bmod{p}}}{\sum_{a=1}^{p-1}\sum_{b=1}^{p-1}}\Leg{\Delta(a,b)}{p}.
\end{align*}
We therefore have
\begin{align*}
\abs{\mathcal{U}}&=(p-1)^2+p^2-3p+1-\Leg{2}{p}+\underset{\substack{a\not\equiv b \bmod{p}\\a+b-2ab\not\equiv 0 \bmod{p}}}{\sum_{a=1}^{p-1}\sum_{b=1}^{p-1}}\Leg{\Delta(a,b)}{p}-(p-3)\\
&=2p^2-6p+5-\Leg{2}{p}+\underset{\substack{a\not\equiv b \bmod{p}\\a+b-2ab\not\equiv 0 \bmod{p}}}{\sum_{a=1}^{p-1}\sum_{b=1}^{p-1}}\Leg{\Delta(a,b)}{p}.
\end{align*}

Finally, we have
\begin{align*}
T(p)&=\abs{\mathcal{U}}-\abs{\mathcal{U}_0}\\
&=2p^2-6p+5-\Leg{2}{p}+\underset{\substack{a\not\equiv b \bmod{p}\\a+b-2ab\not\equiv 0 \bmod{p}}}{\sum_{a=1}^{p-1}\sum_{b=1}^{p-1}}\Leg{\Delta(a,b)}{p}-(9p-20)\\
&=(p-5)(2p-5)-\Leg{2}{p}+\underset{\substack{a\not\equiv b \bmod{p}\\a+b-2ab\not\equiv 0 \bmod{p}}}{\sum_{a=1}^{p-1}\sum_{b=1}^{p-1}}\Leg{\Delta(a,b)}{p}.
\end{align*}
This ends the proof of Theorem \ref{th:22}.

\section{Proof of Theorem \ref{th:2}}

We now finish the proof of the fourth power mean of $H(m,n,k,\chi;p)$. Note that
\begin{align}
&(p-1)\sum_{\chi \bmod{p}} \sum_{m=1}^{p-1} \abs{\sum_{a=1}^{p-1} \chi(ma+n\ov{a})\ee{\frac{ka}{p}}}^4\notag\\
&\quad = \sum_{\chi \bmod{p}}\sum_{\psi \bmod{p}}\abs{\sum_{m=1}^{p-1}\psi(m)\abs{\sum_{a=1}^{p-1} \chi(ma+n\ov{a})\ee{\frac{ka}{p}}}^2}^2\notag\\
&\quad = \sum_{\chi \bmod{p}}\abs{\sum_{m=1}^{p-1}\abs{\sum_{a=1}^{p-1} \chi(ma+n\ov{a})\ee{\frac{ka}{p}}}^2}^2\notag\\
&\quad\quad+\sum_{\chi \bmod{p}}\sum_{\substack{\psi \bmod{p}\\\psi\ne \chi_0}}\abs{\sum_{m=1}^{p-1}\psi(m)\abs{\sum_{a=1}^{p-1} \chi(ma+n\ov{a})\ee{\frac{ka}{p}}}^2}^2.\label{eq:sum1}
\end{align}
Here the first sum in the last identity is obtained by Theorem \ref{th:1}. Hence it suffices to compute the second sum. We seperate the second sum into two parts
\begin{align}
&\sum_{\chi \bmod{p}}\sum_{\substack{\psi \bmod{p}\\\psi\ne \chi_0}}\abs{\sum_{m=1}^{p-1}\psi(m)\abs{\sum_{a=1}^{p-1} \chi(ma+n\ov{a})\ee{\frac{ka}{p}}}^2}^2\notag\\
&\quad = \sum_{\substack{\psi \bmod{p}\\\psi\ne \chi_0}}\abs{\sum_{m=1}^{p-1}\psi(m)\abs{\sum_{a=1}^{p-1} \chi_0(ma+n\ov{a})\ee{\frac{ka}{p}}}^2}^2\notag\\
&\quad\quad+\sum_{\substack{\chi \bmod{p}\\\chi\ne \chi_0}}\sum_{\substack{\psi \bmod{p}\\\psi\ne \chi_0}}\abs{\sum_{m=1}^{p-1}\psi(m)\abs{\sum_{a=1}^{p-1} \chi(ma+n\ov{a})\ee{\frac{ka}{p}}}^2}^2.\label{eq:sum2}
\end{align}

We next prove the following result.
\begin{lemma}\label{le:4.1}
Let $p>3$ be an odd prime. Let $n$ and $k$ be integers coprime to $p$. Let $\psi$ be a non-principal character mod $p$. We have
\begin{align}
&\sum_{m=1}^{p-1}\psi(m)\abs{\sum_{a=1}^{p-1} \chi_0(ma+n\ov{a})\ee{\frac{ka}{p}}}^2\notag\\
&\quad=\psi(-nk^2)\times\begin{cases}
p-4 & \text{if $\psi=\chi_L$},\\
\tau\left(\ov{\psi}^2\right)(2+\psi(4)) & \text{if $\psi$ non-real}.
\end{cases}
\end{align}
\end{lemma}

\begin{proof}
Similar to the first part of the proof of Theorem \ref{th:1}, we have
\begin{align*}
&\frac{1}{\psi(-nk^2)}\sum_{m=1}^{p-1}\psi(m)\abs{\sum_{a=1}^{p-1} \chi_0(ma+n\ov{a})\ee{\frac{ka}{p}}}^2\\
&\ =\sum_{m=1}^{p-1}\psi(m)\abs{\sum_{a=1}^{p-1} \chi_0(ma-\ov{a})\ee{\frac{a}{p}}}^2\\
&\ =\sum_{\substack{m=1\\\Leg{\ov{m}}{p}=-1}}^{p-1}\psi(m)\abs{\sum_{a=1}^{p-1} \ee{\frac{a}{p}}}^2+\sum_{\substack{m=1\\\Leg{\ov{m}}{p}=1}}^{p-1}\psi(m)\abs{\sum_{\substack{a=1\\a^2\not\equiv \ov{m} \bmod{p}}}^{p-1} \ee{\frac{a}{p}}}^2\\
&\ =\sum_{\substack{m=1\\\Leg{m}{p}=-1}}^{p-1}\psi(m)+\sum_{\substack{m=1\\\Leg{m}{p}=1}}^{p-1}\ov{\psi}(m)\abs{\sum_{\substack{a=1\\a^2\not\equiv m \bmod{p}}}^{p-1} \ee{\frac{a}{p}}}^2\\
&\ =\sum_{\substack{m=1\\\Leg{m}{p}=-1}}^{p-1}\psi(m)+\sum_{\substack{m=1\\\Leg{m}{p}=1}}^{p-1}\ov{\psi}(m)\abs{-1-\ee{\frac{m_0}{p}}-\ee{\frac{-m_0}{p}}}^2 \tag{where $m_0^2\equiv m \bmod{p}$}\\
&\ =\sum_{m=1}^{p-1}\frac{1-\Leg{m}{p}}{2}\psi(m)\\
&\ \quad+\frac{1}{2}\sum_{m_0=1}^{p-1}\ov{\psi}^2(m_0)\left(3+2\ee{\frac{m_0}{p}}+2\ee{\frac{-m_0}{p}}+\ee{\frac{2m_0}{p}}+\ee{\frac{-2m_0}{p}}\right)\\
&\ =\begin{cases}
p-4 & \text{if $\psi=\chi_L$},\\
\tau\left(\ov{\psi}^2\right)(2+\psi(4)) & \text{if $\psi$ non-real}.
\end{cases}
\end{align*}
\end{proof}

\begin{corollary}\label{cor:4.1}
Let $p>3$ be an odd prime. Let $n$ and $k$ be integers coprime to $p$. We have
\begin{equation}
\sum_{\substack{\psi \bmod{p}\\\psi\ne \chi_0}}\abs{\sum_{m=1}^{p-1}\psi(m)\abs{\sum_{a=1}^{p-1} \chi_0(ma+n\ov{a})\ee{\frac{ka}{p}}}^2}^2=6p^2-31p+16.
\end{equation}
\end{corollary}

\begin{proof}
By Lemma \ref{le:4.1}, we have
\begin{align*}
&\sum_{\substack{\psi \bmod{p}\\\psi\ne \chi_0}}\abs{\sum_{m=1}^{p-1}\psi(m)\abs{\sum_{a=1}^{p-1} \chi_0(ma+n\ov{a})\ee{\frac{ka}{p}}}^2}^2\\
&\quad=(p-4)^2+p\sum_{\psi\ne \chi_0,\chi_L}\abs{2+\psi(4)}^2\\
&\quad=(p-4)^2+p\sum_{\psi\ne \chi_0,\chi_L}\left(5+2\left(\psi(4)+\ov{\psi}(4)\right)\right)\\
&\quad=(p-4)^2+5p(p-3)-8p\\
&\quad=6p^2-31p+16.
\end{align*}
\end{proof}

We also require
\begin{lemma}\label{le:4.2}
Let $p>3$ be an odd prime. Let $n$ and $k$ be integers coprime to $p$. Let $\chi$ and $\psi$ be non-principal characters mod $p$. We have
\begin{align}
&\sum_{m=1}^{p-1}\psi(m)\abs{\sum_{a=1}^{p-1} \chi(ma+n\ov{a})\ee{\frac{ka}{p}}}^2\notag\\
&\quad=\frac{\psi(nk^2)\tau(\psi)\tau(\ov{\psi})}{p}\notag\\
&\quad\quad\times\begin{cases}
-(p-1)-\sum_{u=1}^{p-1}\ov{\chi}(u)\sum_{a=2}^{p-1}\Leg{(ua-1)(u\ov{a}-1)}{p} & \text{if $\psi=\chi_L$},\\
\tau\left(\ov{\psi}^2\right)\sum_{u=1}^{p-1}\ov{\chi}(u)\sum_{a=1}^{p-1}\ov{\psi}(ua-1)\psi(u\ov{a}-1)\psi^2(a-1) & \text{if $\psi$ non-real}.
\end{cases}
\end{align}
\end{lemma}

\begin{proof}
It follows by \eqref{eq:1.1} that
\begin{align*}
&\frac{1}{\psi(nk^2)}\sum_{m=1}^{p-1}\psi(m)\abs{\sum_{a=1}^{p-1} \chi(ma+n\ov{a})\ee{\frac{ka}{p}}}^2\\
&\quad=\sum_{m=1}^{p-1}\psi(m)\abs{\sum_{a=1}^{p-1} \chi(ma+\ov{a})\ee{\frac{a}{p}}}^2\\
&\quad = \sum_{m=1}^{p-1}\psi(m)\abs{\sum_{a=1}^{p-1} \frac{1}{\tau(\ov{\chi})}\sum_{u=1}^{p-1}\ov{\chi}(u)\ee{\frac{u(ma+\ov{a})}{p}}\ee{\frac{a}{p}}}^2\\
&\quad = \frac{1}{p}\sum_{u=1}^{p-1}\sum_{v=1}^{p-1}\ov{\chi}(u\ov{v})\sum_{a=1}^{p-1}\sum_{b=1}^{p-1}\ee{\frac{a-b+u\ov{a}-v\ov{b}}{p}}\sum_{m=1}^{p-1}\psi(m)\ee{\frac{m(ua-vb)}{p}}\\
&\quad = \frac{\tau(\psi)}{p}\sum_{u=1}^{p-1}\sum_{v=1}^{p-1}\ov{\chi}(u\ov{v})\sum_{a=1}^{p-1}\sum_{b=1}^{p-1}\ov{\psi}(ua-vb)\ee{\frac{a-b+u\ov{a}-v\ov{b}}{p}}\\
&\quad =\frac{\tau(\psi)}{p} \sum_{u=1}^{p-1}\sum_{v=1}^{p-1}\ov{\chi}(u)\sum_{a=1}^{p-1}\sum_{b=1}^{p-1}\ov{\psi}(v)\ov{\psi}(ua-b)\ee{\frac{a-b+v(u\ov{a}-\ov{b})}{p}}\\
&\quad = \frac{\tau(\psi)}{p} \sum_{u=1}^{p-1}\ov{\chi}(u)\sum_{a=1}^{p-1}\sum_{b=1}^{p-1}\ov{\psi}(ua-b)\ee{\frac{a-b}{p}}\sum_{v=1}^{p-1}\ov{\psi}(v)\ee{\frac{v(u\ov{a}-\ov{b})}{p}}\\
&\quad = \frac{\tau(\psi)\tau(\ov{\psi})}{p} \sum_{u=1}^{p-1}\ov{\chi}(u)\sum_{a=1}^{p-1}\sum_{b=1}^{p-1}\ov{\psi}(ua-b)\psi(u\ov{a}-\ov{b})\ee{\frac{a-b}{p}}\\
&\quad = \frac{\tau(\psi)\tau(\ov{\psi})}{p} \sum_{u=1}^{p-1}\ov{\chi}(u)\sum_{a=1}^{p-1}\ov{\psi}(ua-1)\psi(u\ov{a}-1)\sum_{b=1}^{p-1}\ov{\psi}^2(b)\ee{\frac{b(a-1)}{p}}.
\end{align*}

If $\psi$ non-real, then $\ov{\psi}^2$ is non-principal, and hence
\begin{align*}
&\frac{1}{\psi(nk^2)}\sum_{m=1}^{p-1}\psi(m)\abs{\sum_{a=1}^{p-1} \chi(ma+n\ov{a})\ee{\frac{ka}{p}}}^2\\
&\quad = \frac{\tau(\psi)\tau(\ov{\psi})\tau\left(\ov{\psi}^2\right)}{p} \sum_{u=1}^{p-1}\ov{\chi}(u)\sum_{a=1}^{p-1}\ov{\psi}(ua-1)\psi(u\ov{a}-1)\psi^2(a-1).
\end{align*}

If $\psi$ is the Legendre symbol mod $p$, then $\ov{\psi}^2$ is principal, and hence
\begin{align*}
&\frac{1}{\psi(nk^2)}\sum_{m=1}^{p-1}\psi(m)\abs{\sum_{a=1}^{p-1} \chi(ma+n\ov{a})\ee{\frac{ka}{p}}}^2\\
&\quad = \frac{\tau^2(\chi_L)}{p} \sum_{u=1}^{p-1}\ov{\chi}(u)\sum_{a=1}^{p-1}\Leg{(ua-1)(u\ov{a}-1)}{p}\sum_{b=1}^{p-1}\ee{\frac{b(a-1)}{p}}\\
&\quad = \frac{\tau^2(\chi_L)}{p} \left((p-1)\sum_{u=1}^{p-1}\ov{\chi}(u)\Leg{(u-1)^2}{p}-\sum_{u=1}^{p-1}\ov{\chi}(u)\sum_{a=2}^{p-1}\Leg{(ua-1)(u\ov{a}-1)}{p}\right)\\
&\quad = \frac{\tau^2(\chi_L)}{p} \left(-(p-1)-\sum_{u=1}^{p-1}\ov{\chi}(u)\sum_{a=2}^{p-1}\Leg{(ua-1)(u\ov{a}-1)}{p}\right).
\end{align*}
\end{proof}

Before stating two corollaries of Lemma \ref{le:4.2}, I claim two useful sums of the Legendre symbol.

\begin{lemma}\label{le:4.3}
Let $p$ be an odd prime. We have
\begin{equation}\label{eq:4.3a}
\sum_{a=1}^{p-1}\Leg{(a-1)(\ov{a}-1)}{p}=-\Leg{-1}{p}
\end{equation}
and
\begin{equation}\label{eq:4.3b}
\sum_{u=1}^{p-1}\sum_{a=1}^{p-1}\Leg{(ua-1)(u\ov{a}-1)}{p}=2.
\end{equation}
\end{lemma}

\begin{proof}
To prove \eqref{eq:4.3a}, we have
$$\sum_{a=1}^{p-1}\Leg{(a-1)(\ov{a}-1)}{p}=\sum_{a=1}^{p-1}\Leg{-a(a-1)^2}{p}=\sum_{a=2}^{p-1}\Leg{-a}{p}=-\Leg{-1}{p}.$$

To prove \eqref{eq:4.3b}, we recall the following classical result on the sum of Legendre symbol of quadratic polynomials (cf. \cite[Theorem 2.1.2]{BEW1998})
$$\sum_{a=0}^{p-1}\Leg{a^2+ma+n}{p}=\begin{cases}
-1 & \text{if $p\nmid m^2-4n$},\\
p-1 & \text{if $p\mid m^2-4n$}.
\end{cases}$$
Hence
\begin{align*}
&\sum_{u=1}^{p-1}\sum_{a=1}^{p-1}\Leg{(ua-1)(u\ov{a}-1)}{p}\\
&\quad = \sum_{a=\pm 1}\sum_{u=0}^{p-1}\Leg{u^2-(a+\ov{a})u+1}{p}+\sum_{a=2}^{p-2}\sum_{u=0}^{p-1}\Leg{u^2-(a+\ov{a})u+1}{p}-(p-1)\\
&\quad = 2(p-1)-(p-3)-(p-1)=2.
\end{align*}
\end{proof}

Now we show

\begin{corollary}\label{cor:4.2a}
Let $p>3$ be an odd prime. Let $n$ and $k$ be integers coprime to $p$. We have
\begin{align}
&\sum_{\substack{\chi \bmod{p}\\\chi\ne \chi_0}}\abs{\sum_{m=1}^{p-1}\Leg{m}{p}\abs{\sum_{a=1}^{p-1} \chi(ma+n\ov{a})\ee{\frac{ka}{p}}}^2}^2\notag\\
&\quad= (p^3-2p^2-4p-4)-2p(p-1)\Leg{-1}{p}+(p-1)T_L(p),
\end{align}
where $T_L(p)$ is defined in Theorem \ref{th:2}.
\end{corollary}

\begin{proof}
It follows by Lemma \ref{le:4.2} that
\begin{align*}
&\sum_{\substack{\chi \bmod{p}\\\chi\ne \chi_0}}\abs{\sum_{m=1}^{p-1}\Leg{m}{p}\abs{\sum_{a=1}^{p-1} \chi(ma+n\ov{a})\ee{\frac{ka}{p}}}^2}^2\\
&\quad= \sum_{\chi\ne \chi_0}\abs{-(p-1)-\sum_{u=1}^{p-1}\ov{\chi}(u)\sum_{a=2}^{p-1}\Leg{(ua-1)(u\ov{a}-1)}{p}}^2\\
&\quad = \sum_{\chi\ne \chi_0}\Bigg((p-1)^2+(p-1)\sum_{u=1}^{p-1}\left(\chi(u)+\ov{\chi}(u)\right)\sum_{a=2}^{p-1}\Leg{(ua-1)(u\ov{a}-1)}{p}\\
&\quad\quad\quad\quad\quad +\abs{\sum_{u=1}^{p-1}\ov{\chi}(u)\sum_{a=2}^{p-1}\Leg{(ua-1)(u\ov{a}-1)}{p}}^2\Bigg).
\end{align*}

With the help of Lemma \ref{le:4.3}, we derive
\begin{align*}
&(p-1)\sum_{\chi\ne \chi_0}\sum_{u=1}^{p-1}\left(\chi(u)+\ov{\chi}(u)\right)\sum_{a=2}^{p-1}\Leg{(ua-1)(u\ov{a}-1)}{p}\\
&\quad = 2(p-1)\sum_{u=1}^{p-1}\sum_{a=2}^{p-1}\Leg{(ua-1)(u\ov{a}-1)}{p}\sum_{\chi\ne \chi_0}\chi(u)\\
&\quad = 2(p-1)(p-2)\sum_{a=2}^{p-1}\Leg{(a-1)(\ov{a}-1)}{p}-2(p-1)\sum_{u=2}^{p-1}\sum_{a=2}^{p-1}\Leg{(ua-1)(u\ov{a}-1)}{p}\\
&\quad = 2(p-1)(p-2)\left(-\Leg{-1}{p}\right)-2(p-1)\left(-p+4+\Leg{-1}{p}\right)\\
&\quad = 2(p-1)\left(p-4-\Leg{-1}{p}(p-1)\right).
\end{align*}

On the other hand, we have
\begin{align*}
&\sum_{\chi\ne \chi_0} \abs{\sum_{u=1}^{p-1}\ov{\chi}(u)\sum_{a=2}^{p-1}\Leg{(ua-1)(u\ov{a}-1)}{p}}^2\\
&\  =\sum_{\chi \bmod{p}}  \abs{\sum_{u=1}^{p-1}\ov{\chi}(u)\sum_{a=2}^{p-1}\Leg{(ua-1)(u\ov{a}-1)}{p}}^2-\abs{\sum_{u=1}^{p-1}\sum_{a=2}^{p-1}\Leg{(ua-1)(u\ov{a}-1)}{p}}^2\\
&\  = \left(\sum_{u=1}^{p-1}\sum_{v=1}^{p-1}\sum_{a=2}^{p-1}\sum_{b=2}^{p-1}\Leg{(ua-1)(u\ov{a}-1)(vb-1)(v\ov{b}-1)}{p}\sum_{\chi \bmod{p}}\ov{\chi}(u\ov{v})\right)-(p-4)^2\\
&\  = (p-1)\left(\sum_{u=1}^{p-1}\sum_{a=2}^{p-1}\sum_{b=2}^{p-1}\Leg{(ua-1)(u\ov{a}-1)(ub-1)(u\ov{b}-1)}{p}\right)-(p-4)^2\\
&\  = (p-1)\Bigg(\sum_{u=1}^{p-1}\sum_{a=1}^{p-1}\sum_{b=1}^{p-1}\Leg{(ua-1)(u\ov{a}-1)(ub-1)(u\ov{b}-1)}{p}\\
&\ \ \ \quad\quad\quad\quad-2\sum_{u=1}^{p-1}\sum_{a=1}^{p-1}\Leg{(u-1)^2(ua-1)(u\ov{a}-1)}{p}+\sum_{u=1}^{p-1}\Leg{(u-1)^4}{p}\Bigg)-(p-4)^2\\
&\  = (p-1)\left(T_L(p)-2\left(2+\Leg{-1}{p}\right)+(p-2)\right)-(p-4)^2.
\end{align*}

Altogether, we arrive at
\begin{align*}
&\sum_{\substack{\chi \bmod{p}\\\chi\ne \chi_0}}\abs{\sum_{m=1}^{p-1}\Leg{m}{p}\abs{\sum_{a=1}^{p-1} \chi(ma+n\ov{a})\ee{\frac{ka}{p}}}^2}^2\\
&\quad = (p-1)^2 (p-2)+2(p-1)\left(p-4-\Leg{-1}{p}(p-1)\right)\\
&\quad\quad +(p-1)\left(T_L(p)-2\left(2+\Leg{-1}{p}\right)+(p-2)\right)-(p-4)^2\\
&\quad = (p^3-2p^2-4p-4)-2p(p-1)\Leg{-1}{p}+(p-1)T_L(p).
\end{align*}
\end{proof}

\begin{corollary}\label{cor:4.2b}
Let $p>3$ be an odd prime. Let $n$ and $k$ be integers coprime to $p$. We have
\begin{align}
&\sum_{\substack{\chi \bmod{p}\\\chi\ne \chi_0}}\sum_{\substack{\psi \bmod{p}\\\psi\ne \chi_0,\chi_L}}\abs{\sum_{m=1}^{p-1}\psi(m)\abs{\sum_{a=1}^{p-1} \chi(ma+n\ov{a})\ee{\frac{ka}{p}}}^2}^2\notag\\
&\quad= -p(p^4-9p^3+37p^2-76p+29)+2p(p-1)\Leg{-1}{p}\notag\\
&\quad\quad-p(p-1)\,T_L(p)+p(p-1)^2\,T(p),
\end{align}
where $T(p)$ and $T_L(p)$ are defined in Theorem \ref{th:2}.
\end{corollary}

\begin{proof}
Again, we use Lemma \ref{le:4.2} to get
\begin{align*}
&\sum_{\substack{\chi \bmod{p}\\\chi\ne \chi_0}}\sum_{\substack{\psi \bmod{p}\\\psi\ne \chi_0,\chi_L}}\abs{\sum_{m=1}^{p-1}\psi(m)\abs{\sum_{a=1}^{p-1} \chi(ma+n\ov{a})\ee{\frac{ka}{p}}}^2}^2\\
&\quad = p\sum_{\chi\ne \chi_0}\sum_{\psi\ne \chi_0,\chi_L}\abs{\sum_{u=1}^{p-1}\ov{\chi}(u)\sum_{a=1}^{p-1}\ov{\psi}(ua-1)\psi(u\ov{a}-1)\psi^2(a-1)}^2\\
&\quad = p\sum_{\chi\bmod{p}}\sum_{\psi\ne \chi_0,\chi_L}\abs{\sum_{u=1}^{p-1}\ov{\chi}(u)\sum_{a=1}^{p-1}\ov{\psi}(ua-1)\psi(u\ov{a}-1)\psi^2(a-1)}^2\\
&\quad\quad - p\sum_{\psi\ne \chi_0,\chi_L}\abs{\sum_{u=1}^{p-1}\sum_{a=1}^{p-1}\ov{\psi}(ua-1)\psi(u\ov{a}-1)\psi^2(a-1)}^2.
\end{align*}

Note that
\begin{align*}
&\sum_{\psi\ne \chi_0,\chi_L}\abs{\sum_{u=1}^{p-1}\sum_{a=1}^{p-1}\ov{\psi}(ua-1)\psi(u\ov{a}-1)\psi^2(a-1)}^2\\
&\quad = \sum_{\psi \bmod{p}}\abs{\sum_{u=1}^{p-1}\sum_{a=1}^{p-1}\ov{\psi}(ua-1)\psi(u\ov{a}-1)\psi^2(a-1)}^2\\
&\quad\quad- \abs{\sum_{u=1}^{p-1}\sum_{a=2}^{p-1}\chi_0(ua-1)\chi_0(u\ov{a}-1)}^2 -\abs{\sum_{u=1}^{p-1}\sum_{a=2}^{p-1}\Leg{(ua-1)(u\ov{a}-1)}{p}}^2\\
&\quad = (p-1)(p^3-8p^2+29p-53)-\big(p-2+p-3+(p-3)(p-4)\big)^2-(p-4)^2\\
&\quad = p^3-3p^2-4p-12,
\end{align*}
which is deduced by Theorem \ref{th:3} and Lemma \ref{le:4.3}.

On the other hand,
\begin{align*}
&\sum_{\chi\bmod{p}}\sum_{\psi\ne \chi_0,\chi_L}\abs{\sum_{u=1}^{p-1}\ov{\chi}(u)\sum_{a=1}^{p-1}\ov{\psi}(ua-1)\psi(u\ov{a}-1)\psi^2(a-1)}^2\\
&\quad = \sum_{\chi\bmod{p}}\sum_{\psi\bmod{p}}\abs{\sum_{u=1}^{p-1}\ov{\chi}(u)\sum_{a=1}^{p-1}\ov{\psi}(ua-1)\psi(u\ov{a}-1)\psi^2(a-1)}^2\\
&\quad\quad - \sum_{\chi\bmod{p}}\abs{\sum_{u=1}^{p-1}\ov{\chi}(u)\sum_{a=2}^{p-1}\chi_0(ua-1)\chi_0(u\ov{a}-1)}^2\\
&\quad\quad- \sum_{\chi\bmod{p}}\abs{\sum_{u=1}^{p-1}\ov{\chi}(u)\sum_{a=2}^{p-1}\Leg{(ua-1)(u\ov{a}-1)}{p}}^2.
\end{align*}
We have
\begin{align*}
&\sum_{\chi\bmod{p}}\abs{\sum_{u=1}^{p-1}\ov{\chi}(u)\sum_{a=2}^{p-1}\chi_0(ua-1)\chi_0(u\ov{a}-1)}^2\\
&\quad=\sum_{\chi\bmod{p}}\abs{(p-2)+(p-3)\ov{\chi}(-1)+(p-4)\sum_{u=2}^{p-2}\ov{\chi}(u)}^2\\
&\quad=\sum_{\chi\bmod{p}}\Big((p-2)+(p-3)\ov{\chi}(-1)\Big)^2+(p-4)^2\sum_{\chi\bmod{p}}\sum_{u=2}^{p-2}\sum_{v=2}^{p-2}\ov{\chi}(u\ov{v})\\
&\quad\quad +(p-4)\sum_{\chi\bmod{p}}\Big((p-2)+(p-3)\ov{\chi}(-1)\Big)\sum_{u=2}^{p-2}\Big(\chi(u)+\ov{\chi}(u)\Big)\\
&\quad = (p-1)(p-2)^2+(p-1)(p-3)^2+(p-1)(p-3)(p-4)^2\\
&\quad = (p-1)(p^3-9p^2+30p-35).
\end{align*}
Note also that
\begin{align*}
&\sum_{\chi\bmod{p}}\abs{\sum_{u=1}^{p-1}\ov{\chi}(u)\sum_{a=2}^{p-1}\Leg{(ua-1)(u\ov{a}-1)}{p}}^2\\
&\quad=(p-1)\left(T_L(p)-2\left(2+\Leg{-1}{p}\right)+(p-2)\right),
\end{align*}
which is proved in the proof of Corollary \ref{cor:4.2a}. Finally, we have
\begin{align*}
&\sum_{\chi\bmod{p}}\sum_{\psi\bmod{p}}\abs{\sum_{u=1}^{p-1}\ov{\chi}(u)\sum_{a=1}^{p-1}\ov{\psi}(ua-1)\psi(u\ov{a}-1)\psi^2(a-1)}^2\\
&\quad = (p-1)\sum_{\psi\bmod{p}}\sum_{u=1}^{p-1}\sum_{a=1}^{p-1}\sum_{b=1}^{p-1}\Bigg(\ov{\psi}(ua-1)\psi(u\ov{a}-1)\psi^2(a-1)\\
&\quad\quad\quad\quad\quad\quad\quad\quad\quad\quad\quad\quad\quad\quad \times\psi(ub-1)\ov{\psi}(u\ov{b}-1)\ov{\psi}^2(b-1)\Bigg)\\
&\quad = (p-1)^2\ T(p).
\end{align*}

Altogether, we arrive at
\begin{align*}
&\sum_{\substack{\chi \bmod{p}\\\chi\ne \chi_0}}\sum_{\substack{\psi \bmod{p}\\\psi\ne \chi_0,\chi_L}}\abs{\sum_{m=1}^{p-1}\psi(m)\abs{\sum_{a=1}^{p-1} \chi(ma+n\ov{a})\ee{\frac{ka}{p}}}^2}^2\notag\\
&\quad= p(p-1)^2\,T(p)-p(p-1)(p^3-9p^2+30p-35)\\
&\quad\quad -p(p-1)\left(T_L(p)-2\left(2+\Leg{-1}{p}\right)+(p-2)\right)-p(p^3-3p^2-4p-12)\\
&\quad = -p(p^4-9p^3+37p^2-76p+29)+2p(p-1)\Leg{-1}{p}\\
&\quad\quad-p(p-1)\,T_L(p)+p(p-1)^2\,T(p).
\end{align*}
\end{proof}

Finally, according to \eqref{eq:sum1} and \eqref{eq:sum2}, and with the help of Theorem \ref{th:1} and Corollaries \ref{cor:4.1}, \ref{cor:4.2a} and \ref{cor:4.2b}, it follows that
\begin{align*}
&(p-1)\sum_{\chi \bmod{p}} \sum_{m=1}^{p-1} \abs{\sum_{a=1}^{p-1} \chi(ma+n\ov{a})\ee{\frac{ka}{p}}}^4\\
&\quad = (p-1)(p^4-7p^3+17p^2-5p-25)+(6p^2-31p+16)\\
&\quad\quad +(p^3-2p^2-4p-4)-2p(p-1)\Leg{-1}{p}+(p-1)\,T_L(p)\\
&\quad\quad -p(p^4-9p^3+37p^2-76p+29)+2p(p-1)\Leg{-1}{p}\\
&\quad\quad-p(p-1)\,T_L(p)+p(p-1)^2\,T(p).
\end{align*}
Hence
$$\sum_{\chi \bmod{p}} \sum_{m=1}^{p-1} \abs{\sum_{a=1}^{p-1} \chi(ma+n\ov{a})\ee{\frac{ka}{p}}}^4=(p-1)\Big(p^2-10p+37+p\,T(p)-T_L(p)\Big).$$

%\subsection*{Acknowledgements}

\bibliographystyle{amsplain}

\end{document}